\documentclass[11pt,reqno]{amsproc}

\usepackage{amsmath,amsfonts,amssymb,amsthm}
\usepackage[abbrev,lite,nobysame]{amsrefs}
\usepackage{mathrsfs}
\usepackage{graphics,graphicx}
\usepackage[usenames,dvipsnames]{color}
\usepackage{times}
\usepackage{bbm}
\usepackage[margin=1in]{geometry}
\usepackage[colorlinks=true, pdfstartview=FitV, linkcolor=blue, citecolor=blue, urlcolor=blue]{hyperref}

\makeatletter

\renewcommand\subsection{\@startsection{subsection}{2}%
\normalparindent{.5\linespacing\@plus.7\linespacing}{-.5em}
{\normalfont\bfseries}}

\renewcommand\subsubsection{\@startsection{subsubsection}{3}%
\normalparindent{.5\linespacing\@plus.7\linespacing}{-.5em}
{\normalfont\bfseries}}

\def\@tocline#1#2#3#4#5#6#7{\relax
  \ifnum #1>\c@tocdepth 
  \else
    \par \addpenalty\@secpenalty\addvspace{#2}%
    \begingroup \hyphenpenalty\@M
    \@ifempty{#4}{%
      \@tempdima\csname r@tocindent\number#1\endcsname\relax
    }{%
      \@tempdima#4\relax
    }%
    \parindent\z@ \leftskip#3\relax \advance\leftskip\@tempdima\relax
    \rightskip\@pnumwidth plus4em \parfillskip-\@pnumwidth
    #5\leavevmode\hskip-\@tempdima
      \ifcase #1
       \or\or \hskip 1em \or \hskip 2em \else \hskip 3em \fi%
      #6\nobreak\relax
    \dotfill\hbox to\@pnumwidth{\@tocpagenum{#7}}\par
    \nobreak
    \endgroup
  \fi}
\makeatother

\newtheorem{theorem}{Theorem}
\newtheorem{proposition}{Proposition}[section]
\newtheorem{lemma}[proposition]{Lemma}
\newtheorem{corollary}[proposition]{Corollary}

\theoremstyle{definition}
\newtheorem{definition}[proposition]{Definition}

\numberwithin{equation}{section}

\newcommand\e{{\rm e}}

\def\Re{{\rm Re}}

\def\Im{{\rm Im}}

\newcommand\de{{\partial}}

\newcommand{\norm}[1]{\left\lVert #1 \right\rVert}

\newcommand{\abs}[1]{\left\lvert #1 \right\rvert}

\newcommand\bn{{\boldsymbol n}}

\newcommand\bv{{\boldsymbol v}}

\newcommand\bx{{\boldsymbol x}}

\newcommand\cD{{\mathcal D}}

\newcommand\sfe{{\mathsf e}}

\usepackage{mathtools}

\newtheorem{asmp}{Assumption}

\usepackage{esint}
\newcommand{\N}{\mathbb{N}}

\newcommand{\Z}{\mathbb{Z}}
\newcommand{\R}{\mathbb{R}}

\newcommand{\inv}{^{-1}}

\newcommand{\dist}{\mathrm{dist}}
\newcommand{\sign}{\mathrm{sign}}

\newcommand{\pd}[1]{\partial_{#1}}

\newcommand{\td}[1]{\mathrm{d}#1}
\newcommand{\itd}[2][]{\,{#1}\td{#2}}

\newcommand{\Lapl}{\Delta}

\newcommand{\half}{\frac{1}{2}}
\newcommand{\thalf}{\tfrac{1}{2}}
\newcommand{\tendsto}{\longrightarrow}

\newcommand{\idot}[3][]{\langle#2,#3\rangle_{#1}}

\newcommand{\Lp}[1]{L^{#1}}
\newcommand{\Cts}[1]{C^{#1}}

\newcommand{\ball}{\mathcal{B}}

\newcommand{\bigO}{O}

\newcommand{\disc}{D}
\newcommand{\Radius}{R}
\newcommand{\Interval}{[0,\Radius]}
\newcommand{\circlegrp}{\mathbb{T}}
\newcommand{\init}[1]{#1^{in}}

\newcommand{\Ldisc}{\empty} 
\newcommand{\LinftLdisc}{\Lp\infty_z\Lp2_{r,\theta}} 
\newcommand{\LonetLdisc}{\Lp1_z\Lp2_{r,\theta}} 

\newcommand{\cone}{c} 

\newcommand{\nhdE}[1][_{\lambda,\delta}]{E#1}
\newcommand{\nhdEE}[1][_{\lambda,\delta}]{\mathcal{E}#1}

\newcommand{\Vee}{\mathcal{V}^m_{\lambda,\delta}}

\newcommand{\Hlam}{H_\lambda}
\newcommand{\Lam}[1][_{\nu,k}]{\Lambda#1}

\renewcommand{\Re}{\mathrm{Re}}
\renewcommand{\Im}{\mathrm{Im}}

 \mathtoolsset{showonlyrefs=true}

\begin{document}

\title[Enhanced diffusion in flows with rotational symmetry]{Diffusion enhancement and Taylor dispersion for rotationally symmetric flows  in discs and pipes}

\author[M. Coti Zelati]{Michele Coti Zelati}
\address{Department of Mathematics, Imperial College London, London, SW7 2AZ, UK}
\email{m.coti-zelati@imperial.ac.uk}

\author[M. Dolce]{Michele Dolce}
\address{Institute of Mathematics, EPFL, Station 8, 1015 Lausanne, Switzerland}
\email{michele.dolce@epfl.ch}

\author[C. Lo]{Chia-Chun Lo}
\address{Department of Mathematics, King’s College London, The Strand, London, WC2R 2LS, UK}
\email{chia-chun.lo@kcl.ac.uk}



\begin{abstract}
In this note, we study the long-time dynamics of passive scalars driven by rotationally symmetric flows. We focus on identifying precise conditions on the
velocity field in order to prove enhanced dissipation and Taylor dispersion in three-dimensional infinite pipes. As a byproduct of our analysis, we obtain
an enhanced decay for circular flows on a disc of arbitrary radius.
\end{abstract}

\maketitle


\section{Introduction}
\label{sec:intro}
This note considers the evolution of a passive scalar $f$ in a domain $\Omega\subset\mathbb{R}^d$, $d=2,3$, 
that is advected by an external velocity field $\bv:\Omega \to \mathbb{R}^d$ and is undergoing molecular diffusion. 
Our interest is to study quantitatively how the combined effect of diffusion and advection leads to faster time-scales of 
homogenization for $f$ compared to the case when only diffusion is present. The passive scalar  satisfies the  advection-diffusion equation
\begin{align}
  \label{eqn:advection-diffusion}
  \begin{cases}\pd{t}f+\bv\cdot\nabla f=\nu\Lapl f, \qquad \bx\in \Omega, \, t>0,\\
  f|_{t=0}=f^{in}, \qquad \partial_{\bn} f|_{\partial \Omega}=0.
\end{cases}
\end{align}
where $\nu>0$ is the diffusion coefficient and $\bn$ is the outward unit normal to $\de\Omega$. We are interested in the regime 
where $\nu \ll1 $, in which dissipative effects are observed on large time-scales of order $O(\nu^{-1})$. Since the average over 
the domain is conserved, we always assume that $\fint_{\Omega} f \itd{\bx}=0$. The domain $\Omega$ will either be a 
disc of radius $R>0$, denoted by $D$, or the infinite pipe $D\times \mathbb{R}$, and the velocity field is respectively 
\begin{equation}
\label{eq:defv}
    \bv= rv(r) \hat{\sfe}_{\theta}, \quad \text{in } D,\qquad \bv= v(r)\hat{\sfe}_{z}\quad \text{in } D\times\mathbb{R}.
\end{equation}
Here $\hat{\sfe}_{\theta}$ and $\hat{\sfe}_{z}$ are the unit vectors in the angular direction in $D$ and in the vertical direction in $D\times\mathbb{R}$, respectively. 
This situation and similar have been recently studied in \cites{FMN23,gallay-cz,CZD20}, in analogy with the case of passive scalars advected by shear flows \cites{BCZ17,ABN22}.

The purpose of this short note is twofold: one the one hand, we  identify precise conditions on the velocity field in the pipe 
setting that guarantee the \emph{enhanced dissipation} \cite{CKRZ} and \emph{Taylor dispersion} \cites{Ta1,Ta2,Ar} mechanisms. On the other hand, we derive 
a \textit{dissipation enhancement} result in the disc for a general class of radial velocity fields. In particular, we will assume the following 
for the profile $v(r)$ of the velocity field in \eqref{eq:defv}.
\begin{asmp}
  \label{asmp:nondegeneracy}
  The first $m$ derivatives of $v:\Interval\to\R$ do not vanish
  simultaneously; that is,
  \begin{equation}
    \sum_{n=1}^m\abs{v^{(n)}(r)}\ne 0
  \end{equation}
  for every $0\le r\le \Radius$.
\end{asmp}
We now state our main results, considering differently the pipe and the disc cases.

\subsection{Pipe parallel flows}
\label{subsec:Pipe}
When $\Omega=D\times \mathbb{R}$, the equation \eqref{eqn:advection-diffusion} can be written in cylindrical coordinates as
\begin{equation}
  \label{eqn:cylinder-physical}
  \begin{cases}
  \displaystyle\pd{t}f+v(r)\pd{z}f
  =
  \nu
  \left(
  \frac{1}{r}\pd{r}(r\pd{r})
  +
  \frac{1}{r}\pd{\theta}^2
  +
  \pd{z}^2
  \right)
  f,&(r,\theta,z)\in [0,R]\times \mathbb{T}\times \mathbb{R}, \\
    f\vert_{t=0}=\init{f}, \qquad  \partial_{\bn} f|_{\partial \Omega}=0.
  \end{cases}
\end{equation}
Taking the partial Fourier transform along the
axial coordinate $z$ on both sides of the equation, we see then that
for each $k\in\R$ the Fourier component
\begin{equation}
  \label{eqn:fhatk}
  \hat{f}_k(t,r,\theta)=\int_\R f(t,r,\theta,z)\e^{-ikz}\itd{z}
\end{equation}
satisfies the equation
\begin{equation}
  \label{eqn:cylinder-fourier-1}
  \pd{t}\hat{f}_k+ikv(r)\hat{f}_k
  =
  \nu
  \big( \Delta_{r,\theta}-
  k^2
  \big)
  \hat{f}_k,
\end{equation}
where we denote 
\begin{equation}
\label{def:Deltart}
    \Delta_{r,\theta}:=\frac{1}{r}\pd{r}(r\pd{r})
    +
    \frac{1}{r^2}\pd{\theta}^2.
\end{equation}
In particular, 
each Fourier mode $\hat{f}_k$ evolves independently from all the others,
so in the following, it suffices to consider the
equation~\eqref{eqn:cylinder-fourier-1} for a fixed parameter $k$. A further reduction can be made by considering the function
\begin{equation}
    g_k=\e^{\nu k^2t}\hat{f}_k
\end{equation} 
that satisfies
\begin{equation}
  \label{eqn:cylinder-cov}
  \pd{t}g_k+ikv(r)g_k=\nu\Lapl_{r,\theta}g_k.
\end{equation}
Notice that $g_k$ already incorporates the diffusion along the channel. Our first main result is the following.
\begin{theorem}
  \label{thm:main}
  Let $v:[0,R]\to \mathbb{R}$ satisfy 
  Assumption~\ref{asmp:nondegeneracy} and let $k\neq 0$. 
  Then, there exist constants $c_1,C_1>0$, independent of $\nu,k$, such that for all initial data 
  $g_k^{in}\in L^2(D)$ the solution to \eqref{eqn:cylinder-cov} satisfies 
  \begin{equation}
    \label{eqn:main}
    \norm{g_k(t)}_{L^2(D)}\le
    C_1 \e^{-c_1{\Lam}t} \norm{\init{g}_k}_{L^2(D)}
    \quad\text{where}\quad
    \Lam=
    \begin{cases}
      \nu^{\frac{m}{m+2}}\abs{k}^{\frac{2}{m+2}},&\text{if $0<\nu\le\abs{k}$,}\\
      \frac{k^2}{\nu},&\text{if $0<\abs{k}\le{\nu}$},
    \end{cases}
  \end{equation}
  for every $t\geq0$.
  For the solution to
  \eqref{eqn:cylinder-fourier-1} with initial data $g_k^{in}=\hat{f}^{in}_k$, we have the estimate
  \begin{equation}
    \label{eqn:cor}
    \|\hat{f}_k(t)\|_{\Lp2(\disc)}\le
    C_1 \e^{-(\nu k^2+c_1{\Lam})t}\|\init{\hat{f}_k}\|_{\Lp2(\disc)},
  \end{equation}
  for every $t\geq0$.
\end{theorem}
The bounds in the theorem above are analogous to the ones obtained in \cite{gallay-cz} 
for multi-dimensional shear flows. In fact, they also proved the same result as in Theorem 
\ref{thm:main} for the velocity field $v(r)=1-r^m$ in the disc of radius $1$. However, a general 
condition analogous to that required in Assumption \ref{asmp:nondegeneracy} was not identified. The proof of Theorem \ref{thm:main}, given in Section \ref{sec:Resolvent},
 is inspired by the arguments in \cite{gallay-cz}. In particular, we obtain \eqref{eqn:main} as a consequence of a 
pseudospectral lower bound and the application of a result of Wei \cite{wei}*{Theorem 1.3}, see also Theorem \ref{thm:wei} below.

Having at hand the $k$ by $k$ estimate \eqref{eqn:main}, we are also able to quantify precisely the time-decay for the solution to the original problem \eqref{eqn:cylinder-physical}.
\begin{theorem}
  \label{thm:cylinder-physical}
  Let $f^{in}\in L^1_zL^2_{r,\theta}$, and let
  $v:\Interval\to\R$ satisfy   Assumption~\ref{asmp:nondegeneracy}. Then, there exist constants $c_2,C_2>0$, independent of 
  $\nu$, such that the solution to \eqref{eqn:cylinder-physical} satisfies
  \begin{equation}
    \label{eqn:cylinder-physical-estimate}
    \norm{f(t)}_{\LinftLdisc}
    \le
    C_2
    \left(
    \sqrt{\frac{\nu}{t}}
    +
    \frac{\e^{- c_2\nu t}}{t}
    \right)
    \norm{\init{f}}_{\LonetLdisc}.
  \end{equation}
for every $t\geq 0$.
\end{theorem}
With  \eqref{eqn:cylinder-physical-estimate} we have a precise quantification of the Taylor dispersion mechanism for the problem at hand, see also \cite{bedrossian2022taylor} where these types of bounds were obtained in another context. The polynomial decay rate is analogous to the standard heat equation with diffusivity coefficient $\nu^{-1}$. In particular, the presence of the advection allows us to prove the polynomial decay on a time-scale  $O(1)$, which is much faster than $O(\nu^{-1})$ we would get without advection. 

The choice of the norms on which  to quantify  the decay is rather natural from the available $k$ by $k$ bounds. From a physical point of view,  the flow is stretching the concentration towards spatial infinity in the $z$-direction. Combining this with enough integrability in $z$,   the stretching generated by the flow makes the concentration intersect smaller sets in the discs orthogonal to $z$, so that a decay can be effectively quantified even if the diffusion is still not efficient on a time-scale of order $O(1)$. Notice that the order of the critical points of $v$ does not enter at all in the physical space estimates, contrary to the disc setting, as we show below.
\subsection{Circular flows in a disc}
\label{subsec:circular}
When we consider the equation \eqref{eqn:advection-diffusion} with $\Omega=D$ and $\bv=rv(r)\hat{\sfe}_\theta$, the problem we have at hand is
\begin{equation}
  \label{eqn:disc-physical}
  \begin{cases}
    \pd{t}f+v(r)\pd{\theta}f=\nu\Lapl_{r,\theta}f, \qquad (r,\theta)\in [0,R]\times \mathbb{T},\\
    f\vert_{t=0}=\init{f}, \qquad \partial_{\bn}f|_{\partial \Omega}=0,
  \end{cases}
\end{equation}
where we recall that $\Delta_{r,\theta}$ is defined in \eqref{def:Deltart}.
If we now take a partial Fourier transform in the angular direction, namely 
\begin{equation}
  \label{eqn:f-hat-def}
  \hat{f}_{\ell}=\frac{1}{2\pi}\int_\circlegrp f(t,r,\theta)\e^{-i\ell\theta}\itd{\theta}
\end{equation}
the Fourier coefficients
$\hat{f}_\ell$ of a solution $f$ to equation \eqref{eqn:disc-physical} satisfy%
\begin{equation}
  \label{eqn:disc-fourier}
  \pd{t}\hat{f}_\ell+i\ell v(r)\hat{f}_\ell=\nu\left(\frac{1}{r}\partial_r(r\partial r)-\frac{\ell^2}{r^2}\right)\hat{f}_\ell
\end{equation}
for each $\ell\in\Z$. The analogy with \eqref{eqn:cylinder-cov} is the following: if we take the angular Fourier transform in \eqref{eqn:cylinder-cov}, 
we get 
\begin{equation}
  \label{eqn:cylinder-fourier-2}
  \pd{t}\hat{g}_{k,\ell}
  +
  ik v(r)\hat{g}_{k,\ell}
  =
  \nu\left(
  \frac{1}{r}\pd{r}(r\pd{r})
  -\frac{\ell^2}{r^2}\right)\hat{g}_{k,\ell}.
\end{equation}
Hence, 
\eqref{eqn:disc-fourier} is the equation
\eqref{eqn:cylinder-fourier-2} for the choice of parameters
$\ell=k$.  We can therefore recover the bounds on $\hat{f}_\ell$ from the ones we have for $g_{\ell,\ell}$  in Theorem \ref{thm:main}. Observing that $|\ell|\geq 1>\nu$, we obtain the following.
\begin{corollary}
  \label{cor:disc-fourier}
  Let $v:[0,R]\to \mathbb{R}$ satisfy Assumption~\ref{asmp:nondegeneracy} and $\ell\neq 0$. Then, there exist constants $c_3,C_3>0$, independent of $\nu,\ell$, such that for all initial data $\hat{f}_{\ell}^{in}\in L^2(0,R)$  the solution to   \eqref{eqn:disc-fourier} satisfy
  \begin{equation}
    \|\hat{f}_\ell(t)\|_{\Lp2(0,R)}\le C_3 \e^{-c_3\nu^{\frac{m}{m+2}}\abs{\ell}^{\frac{2}{m+2}}t}\|\init{\hat{f}}_\ell\|_{\Lp2(0,R)},
  \end{equation}
 for every $t\geq 0$. 
\end{corollary}
%
The bound in the physical space for $f$ now it directly follows by Parseval's identity. Namely, if 
\begin{equation}
    \int_{\mathbb{T}} f^{in}(r,\theta) \itd{\theta}=0,
\end{equation}
we obtain that
\begin{equation}
  \label{eqn:parseval}
  \|f(t)\|_{\Lp2(D)}\leq C_4 \e^{-c_3\nu^{\frac{m}{m+2}}t}
    \|{\init{f}}\|_{\Lp2(D)},
\end{equation}
for a suitable constant $C_4>0$. Therefore we capture the enhanced dissipation mechanism, telling us that the solution is decaying on a time-scale $O(\nu^{-\frac{m}{m+2}})$, which is always faster than $O(\nu^{-1})$. When the angular average, corresponding to $\hat{f}_0$, is not zero, we would obtain that $f$ is converging towards its angular average on the fast time-scale. Notice that $\hat{f}_0$ is not conserved but it satisfies a standard $1d$ heat equation, therefore we cannot expect to have decay on a faster time-scale for it.  

\section{Semigroup decay via resolvent estimates}
\label{sec:Resolvent}

The main tool we will employ in the proof of Theorem \ref{thm:main} is a quantitative version of the Gearhart-Pr\"uss obtained by Wei in \cite[Theorem 1.3]{wei} (see also \cite{helffer2010resolvent}), which we reproduce below for the reader's convenience.
\begin{theorem}
  \label{thm:wei}
  Let $X$ be a Hilbert space and $H:\cD(H)\to X$ be an $m$-accretive operator on $X$. Then
  \begin{equation}
    \label{eqn:wei}
    \norm{\e^{-tH}}_{X\to X}\le \e^{-t\Psi(H)+\pi/2}
  \end{equation}
  in which the quantity $\Psi(H)$ is the pseudospectral abscissa of $H$, defined as
  \begin{equation}
    \Psi(H)=\inf \{\norm{(H-z)f}_X\mid z\in i\R, f\in \cD(H), \norm{f}_X=1\}.
  \end{equation}
\end{theorem}
We rewrite the equation \eqref{eqn:cylinder-cov} as
\begin{equation}
  \pd{t}g_k+Hg_k=0,
\end{equation}
where the operator $H:\cD(H)\to\Lp2(\disc)$ is defined as
\begin{equation}
\label{def:H}
  H=-\nu\Lapl_{r,\theta}+ikv(r), \qquad \cD(H)= H^2(D)
\end{equation}
which is indeed $m$-accretive \cite{FMN23}.
Thus, by the Lumer-Phillips theorem, the
unique mild solution to equation \eqref{eqn:cylinder-cov} is given by
a strongly continuous semigroup in $\Lp2(\disc)$, namely, for the
initial datum $\init{g_k}\in\Lp2(\disc)$,
\begin{equation}
  \label{eqn:semigroup-soln}
  g_k(t)=\e^{-tH}\init{g_k}
\end{equation}
solves equation \eqref{eqn:cylinder-cov}. Moreover, by Theorem
\ref{thm:wei}, the operator $\e^{-tH}$ satisfies the estimate
\eqref{eqn:wei}. Hence, the proof of Theorem \ref{thm:main} is reduced in proving a pseudospectral bound for the operator $H$ defined in \eqref{def:H}.

\subsection{Pseudospectral bounds}
\label{subsec:pseudo}
Being $k$ a fixed parameter from now on, let us write
\begin{equation}
  \label{eqn:Hlam-def}
  \Hlam=H-ik\lambda=-\nu\Lapl_{r,\theta}+ik(v(r)-\lambda).
\end{equation}
To prove Theorem~\ref{thm:main}, it suffices to show that
\begin{equation}
  \label{eqn:Hlamg}
  \norm{\Hlam g}_{\Lp2(\disc)}\ge c_1\Lam\norm{g}_{\Lp2(\disc)},
\end{equation}
for every $g\in \cD(H)$, where the constant $c_1$ needs to be chosen
uniformly in $\lambda\in\R$. 

 In the computations that follow, we frequently omit the subscripts on
the notation for norms and inner products in $L^2(D)$, where no ambiguity can occur
as to the relevant function space.  

The strategy for proving \eqref{eqn:Hlamg} is as follows: we choose,
for each $\lambda\in\R$, a neighbourhood
$\nhdEE[_\lambda]\subseteq\Interval$ of the level set
$E_\lambda=v\inv(\lambda)$, and split the domain of integration
\begin{equation}
  \label{eqn:split}
  \norm{g}^2=
  \int_{\abs{x}\in\Interval\setminus\nhdEE[_\lambda]}\abs{g}^2\itd{x}
  +
  \int_{\abs{x}\in\nhdEE[_\lambda]}\abs{g}^2\itd{x}.
\end{equation}
in order to get upper bounds for each of the two integrals on the
right-hand side. The motivation behind this is that, away from the
annulus $\{x\in\disc:\abs{x}\in\nhdEE[_\lambda]\}$, the convection term $v-\lambda$ in \eqref{eqn:Hlam-def} is
bounded away from zero, which allows us to recover a bound on the $L^2$ norm in terms of $H_{\lambda}$ thanks to the invertibility of $v-\lambda$. On the other hand, in the integral over the
region where $\abs{x}\in\nhdEE[_\lambda]$, we  exploit some Poincaré-type  inequality where we can gain smallness parameters from the measure of the set $\nhdEE[_\lambda]$. That the latter set is indeed small is consequence of Assumption~
\ref{asmp:nondegeneracy}.
We thus choose the sets $\nhdEE[_\lambda]$ as follows.
\begin{definition}
\label{defsets}
  Let $m\in\N$ be the one in Assumption~\ref{asmp:nondegeneracy}. Define
  \begin{itemize}
  \item
    $\nhdE$ to be the preimage under $v$ of the interval
    $(\lambda-\delta^m,\lambda+\delta^m)$, that is,
    $\nhdE=\{r\in\Interval\mid \abs{v(r)-\lambda}<\delta^m\}$.
  \item
    $\nhdEE$ to be the neighbourhood of the set $\nhdE$ with
    thickness $\delta^m$, that is, $\nhdEE=\{r\in\Interval\mid
    \dist(x, \nhdE)<\delta^m\}$.
  \end{itemize}
\end{definition}
We collect in the next two propositions the bounds we have for the two integrals on the right-hand side of   \eqref{eqn:split}. Away from the level sets we have the following result.
\begin{proposition}
  \label{prop:away}
Let $\mathcal{E}_{\lambda,\delta}$ be the set defined in Definition \ref{defsets}. Then, for any $g\in \cD(H)$ the following holds true
\begin{equation}
\label{bd:away}
\int_{|x|\in [0,R]\setminus\mathcal{E}_{\lambda,\delta}}|g|^2 \itd{x}\leq \frac14 \|g\|^2+\left( \frac{1}{|k|\delta^m}+\frac{\nu}{|k|^2\delta^{2m+2}}\right)\|H_\lambda g\|\|g\|.
\end{equation}
\end{proposition}
Near the level sets, we can prove the result below.
\begin{proposition}
  \label{prop:near}
Let $\mathcal{E}_{\lambda,\delta}$ be the set defined in Definition \ref{defsets}. Then there exists a constant $\tilde{C}>0$ such that, for any $g\in \cD(H)$,  the following holds true
\begin{equation}
\label{bd:near}
\int_{|x|\in\mathcal{E}_{\lambda,\delta}}|g|^2 \itd{x}\leq \frac12 \|g\|^2+\frac{\tilde{C}\delta^2}{\nu}\|H_\lambda g\|\|g\|.
\end{equation}
\end{proposition}
We postpone the proof of Propositions \ref{prop:away}--\ref{prop:near} to the end of this section. With the bounds \eqref{bd:away} and \eqref{bd:near} at hand, we are ready to present the proof of Theorem \ref{thm:main}.
\begin{proof}[Proof of Theorem~\ref{thm:main}]
  Summing together \eqref{bd:away} and \eqref{bd:near} and
  rearranging, we have for all $\lambda\in\R$ and $\delta>0$ that
  \begin{equation}
    \norm{g}^2
    \le
    4
    \left(
    \frac{1}{\abs{k}\delta^m}
    +
    \frac{\nu}{\abs{k}^2\delta^{2m+2}}
    +
    \frac{\tilde{C}\delta^2}{\nu}
    \right)
    \norm{\Hlam g}\norm{g}.      
  \end{equation}
  We now make a choice of $\delta$ depending on the parameters $\nu$ and $k$:
  \begin{itemize}
  \item
    If $0<\nu\le\abs{k}$, then the sharpest bound we can recover is by
    choosing
    $\delta=\tilde{\delta}\nu^{\frac{1}{m+2}}\abs{k}^{-\frac{1}{m+2}}$,
    resulting in $\norm{g}^2\le
    c_1\nu^{-\frac{m}{m+2}}\abs{k}^{-\frac{2}{m+2}}
    \norm{\Hlam g}\norm{g}$ with the constant $c_1=4(
    \tilde{\delta}^{-m}+\tilde{\delta}^{-(2m+2)}+\tilde{C}\tilde{\delta}^2 )$.
  \item
    If instead $0<\abs{k}\le\nu$, observing that
    \begin{equation}
      \frac{1}{\abs{k}\delta^m}
      +
      \frac{\nu}{\abs{k}^2\delta^{2m+2}}
      +
      \frac{\tilde{C}\delta^2}{\nu}
      =
      \frac{\nu}{k^2}\Bigg(
      \frac{\abs{k}}{\nu}\frac{1}{\delta^m}
      +
      \frac{1}{\delta^{2m+2}}
      +
      \frac{k^2}{\nu^2}\tilde{C}\delta^2
      \Bigg).
    \end{equation}
    Since $|k|/\nu \leq 1$,    we choose $\delta=\tilde{\delta}$, and find that $\norm{g}^2\le
    c_1\frac{\nu}{k^2}\norm{\Hlam g}\norm{g}$, with
    the same constant $c_1$ as in the previous case.
  \end{itemize}
  Altogether we recover inequality \eqref{eqn:Hlamg}, thanks to which we can apply Theorem \ref{thm:wei} and conclude 
  the proof of Theorem~\ref{thm:main}.
\end{proof}
It thus remains to show the proofs of Proposition  \ref{prop:away}-\ref{prop:near}, which we present in the next two sections.

\subsection{Bounds away from level sets}
\label{subsec:away}
In this section, we aim at proving Proposition \ref{prop:away}. To this end, we follow the strategy in~\cite{gallay-cz}, and we introduce the function
\begin{equation}
  \chi(r)=\varphi(\sign(v(r)-\lambda)\dist(r,\nhdE)/\delta^m),
\end{equation}
in which
\begin{equation}
  \varphi(s)=\begin{cases}
  s,&\text{if $\abs{s}\le 1$},\\
  \sign(s),\quad&\text{otherwise.}
  \end{cases}
\end{equation}
We are now ready to prove Proposition \ref{prop:away}.
\begin{proof}[Proof of Proposition \ref{prop:away}]
By the definition of $\chi$, we know that $\chi(v-\lambda)\geq 0$. Moreover, in the set $\mathcal{E}_{\lambda,\delta}$ we have $|v-\lambda|\geq \delta^m$. Therefore
\begin{equation}
  \label{eqn:away-1}
  \int_{\abs{x}\in\Interval\setminus{\nhdEE}}\abs{g}^2\itd{x}
 \le \int_\circlegrp\int_0^\Radius\frac{(v(r)-\lambda)\chi(r)}{\delta^m}\abs{g(r,\theta)}^2 r\itd{r}\itd{\theta}
=
  \frac{1}{\delta^m}
  \idot[\Ldisc]{(v-\lambda)\chi g}{g}.
\end{equation} 

To estimate the term $\idot[\Ldisc]{(v-\lambda)\chi g}{g}$, observe
that
\begin{equation}
  \label{eqn:away-2}
  \begin{aligned}
    \norm{\Hlam g}
    \norm{g}
    &\ge
    \Im{\idot{\Hlam g}{\chi g}}\\
    &=
    \nu\Im{\idot{\Lapl_{r,\theta}g}{\chi g}}
    +
    \Im{\idot{ik(v-\lambda)g}{\chi g}}\\
    &=
    -\nu
    \Im{\idot{\partial_rg}{(\partial_r\chi)g}}+
    k\idot{(v-\lambda)g}{\chi g},
  \end{aligned}
\end{equation}
in which we recognise the relevant term on the final line. Noting that
$\abs{\pd{r}\chi}<\delta\inv$, it then
follows from \eqref{eqn:away-2} and the triangle inequality that
\begin{equation}
  \label{eqn:away-3}
    \idot{(v-\lambda)\chi g}{g}
    \le
    \frac{1}{\abs{k}}\left(
    \norm{\Hlam g}
    \norm{g}
    +\frac{\nu}{\delta}
    \norm{\nabla g}
    \norm{g}
    \right).
\end{equation}
Observe also that
\begin{equation}
  \label{eqn:away-4}
  \nu\norm{\nabla g}^2=\Re\idot{\Hlam g}{g}\le\norm{\Hlam g}\norm{g}.
\end{equation}
Thus, combining \eqref{eqn:away-1} with \eqref{eqn:away-3} and \eqref{eqn:away-4}, we have
\begin{equation}
  \label{eqn:away-5}
  \begin{aligned}
    \int_{\abs{x}\in\Interval\setminus{\nhdEE}}
    \abs{g}^2
    \itd{x}
    &\le
    \frac{1}{\delta^m}
    \idot{(v-\lambda)\chi g}{g}\\
    &\le
    \frac{1}{\abs{k}\delta^m}
    \left(
    \norm{\Hlam g}
    \norm{g}
    +\frac{\nu}{\delta}
    \norm{\nabla g}
    \norm{g}
    \right)\\
    &\le
    \frac{1}{\abs{k}\delta^m}
    \left(
    \norm{\Hlam g}
    \norm{g}
    +\frac{\nu^\half}{\delta}
    \norm{\Hlam g}^\half
    \norm{g}^{\frac{3}{2}}
    \right)\\
    &\le
    \left(
    \frac{1}{\abs{k}\delta^m}
    +
    \frac{\nu}{\abs{k}^2\delta^{2m+2}}
    \right)
    \norm{\Hlam g}
    \norm{g}  
    +
    \frac{1}{4}
    \norm{g}^2,
  \end{aligned}
\end{equation}
where we also applied the Young's inequality on the product
$\nu^\half\delta^{-(m+1)} \norm{\Hlam g}^\half
\norm{g}^\half$ on the penultimate line.
\end{proof}

\subsection{Bounds near level sets}
\label{subsec:near}
To prove Proposition \ref{prop:near}, we  use two results from
\cite{gallay-cz}. The first of these is a Poincaré-type bound which
appears in \cite[Lemma B.1]{gallay-cz}:
\begin{lemma}
  \label{lem:B2}
  For all $g\in H^1(D)$ and all $R\geq R_2\ge R_1\ge0$, we have 
  \begin{equation}
    \int_{R_1\le\abs{x}\le R_2}|g|^2\itd{x}
    \le
    2(R_2-R_1)\norm{g}\norm{\nabla g} .
  \end{equation}
\end{lemma}
The second result is that $\nhdEE(v)$ is covered by a finite union of
intervals whose total length is in $\bigO(\delta)$ as $\delta\tendsto
0$.  
\begin{lemma}
  \label{lem:covering}
  Let $v\in\Cts{m}(\Interval)$ satisfy
  Assumption~\ref{asmp:nondegeneracy}. Then there exist constants
  $C_0,\delta_0>0$ and, for each $\lambda\in\R$ and $\delta>0$, a
  choice of a finite family $\Vee$ of intervals such that
  \begin{equation}
    \nhdEE(v)\subseteq\bigcup\Vee\subseteq\Interval
  \end{equation}
  and such that for all $\lambda\in\R$ and $0<\delta\le\delta_0$ we
  have that
  \begin{equation}
    \label{eqn:covering-length}
    \sum_{V\in\Vee}\abs{V}<C_0\delta.
  \end{equation}
\end{lemma}
\begin{proof}
  This Lemma can be extracted from the proof of 
  \cite[Lemma 2.6]{gallay-cz}, where such coverings by intervals are constructed in
  order to bound the measure $\abs{\nhdEE}$ of the level set
  neighbourhoods.

  Observe first that it suffices to prove this result with $\nhdE$ in
  place of $\nhdEE$, since one may enlarge by $\delta$ each interval in
  a covering of $\nhdE$ to produce one for $\nhdEE$ that still satisfies
  \eqref{eqn:covering-length} but with a worse constant
  $C_0$. Moreover, since $\nhdE$ is empty for $\lambda$ outside of a
  compact neighbourhood of {$v(\Interval)\subseteq\R$}, it suffices to be
  able to choose $C_0$ locally constant near each $\lambda_0$ in this
  neighbourhood.

  Fix now $\lambda_0\in \R$. The idea is to use the fact that $\nhdE$ is
  a union of level sets $E_\lambda$ with $\lambda$ close to $\lambda_0$:
  \begin{equation}
    \nhdE=\bigcup_{\abs{\lambda-\lambda_0}<\delta^m}E_\lambda
  \end{equation}
  and by the continuity of the function $v$, we expect the level set
  $E_\lambda$ not too change to much when $\lambda$ is perturbed away
  from $\lambda_0$. Indeed, using Assumption~\ref{asmp:nondegeneracy},
  $E_{\lambda_0}=v\inv(\lambda_0)$ consists of finitely many elements
  $r_1,\dots,r_{N_{\lambda_0}}$. Near each $r_i$, the function $v$ is
  approximated by its Taylor series
  \begin{equation}
    v(r)\approx \lambda+a_i(r-r_i)^{n_i}
    \quad\text{where}\quad
    a_i=\frac{v^{(n_1)}(r_i)}{n_i!},
  \end{equation}
  in which $n_i\in\N$ is the order of the lowest-order derivative of
  $v$ which does not vanish at $r_i$; again by Assumption~\ref{asmp:nondegeneracy} we have that $1\le n_i\le m$.  For small
  $\delta>0$, the function $v$ then approximately maps the interval
  $\ball_{\delta}(r_i)\subseteq\Interval$ to the interval
  $\ball_{a_i\delta^{n_i}}(\lambda)\subseteq\R$. Conversely, one is
  able to choose $R_0>0$ such that
  \begin{equation}
    \label{eqn:interval-cover}
    v\inv(\ball_{\delta^{m}}(\lambda))\subseteq
    \bigcup_{r_i\in E_\lambda}\ball_{R_0\delta}(r_i)
  \end{equation}
  for all sufficiently small $\delta>0$.  Once again we refer to the
  reference~\cite{gallay-cz} for the details of this computation.
  
  Choose now $V_i=\ball_{R_0\delta}(r_i)$ for
  $i=1,\dots,N_{\lambda_0}$. Then \eqref{eqn:interval-cover} is
  precisely the statement that the collection
  $\{V_i\}_{i=1}^{N_{\lambda_0}}$ covers $\nhdE$ for all $\lambda$ in
  a small neighbourhood around near $\lambda_0$. Finally, choosing
  $C_0=2N_{\lambda_0}R_0$, we find that \eqref{eqn:covering-length} is
  satisfied:
  $\sum_{i=1}^{N_{\lambda_0}}\abs{V_i}=2N_{\lambda_0}R_0\delta=C_0\delta$.
\end{proof}
With the covering  by intervals $\Vee$ just obtained in Lemma \ref{lem:covering}, we are to prove Proposition \ref{prop:near}.

\begin{proof}[Proof of Proposition \ref{prop:near}]
First, we observe that for any $V\in \Vee$, thanks to Lemma \ref{lem:B2} we have 
\begin{equation}
\int_{\abs{x}\in V}\abs{g}^2\itd{x}\le 2\abs{V}\norm{g}\norm{\nabla g}.
\end{equation}
Therefore,
\begin{equation}
  \begin{aligned}
    \int_{\abs{x}\in\nhdEE[]}
    \abs{g}^2
    \itd{x}
    &\le
    \sum_{V\in\Vee}
     \int_{\abs{x}\in V}\abs{g}^2\itd{x}
\le
    \norm{g}\norm{\nabla g}
    \sum_{V\in\Vee}
    2\abs{V} \\
    &\le
    2C_0\delta\norm{g}\norm{\nabla g}\le
    \half\norm{g}^2
    +
    2{C_0}^2\delta^2
    \norm{\nabla g}^2,
  \end{aligned}
\end{equation}
where we have used Lemma~\ref{lem:covering},
followed by Young's inequality on the final line. Combining this 
with \eqref{eqn:away-4} we find that
\begin{equation}
  \label{eqn:near-1}
    \int_{\abs{x}\in\nhdEE[]}
    \abs{g}^2
    \itd{x}
    \le
    \half\norm{g}^2+
    \frac{2C_0^2\delta^2}{\nu}\norm{\Hlam g}\norm{g}.
\end{equation}
\end{proof}

\section{Estimates in physical space}
\label{sec:physical}

In this Section we prove Theorem~\ref{thm:cylinder-physical}, which
gives a decay estimate on the $\LinftLdisc$ norm of a solution to
\eqref{eqn:cylinder-physical} when the initial datum belongs to $\LonetLdisc$.

\begin{proof}[Proof of Theorem~\ref{thm:cylinder-physical}]
    Using that the Fourier transform is a continuous map between $\Lp
    1$ and $\Lp \infty$ together with H\"older's inequality, thanks to Theorem \ref{thm:main} we have
    that
    \begin{equation}
      \begin{aligned}
        \norm{f(t)}_{\LinftLdisc}
        &\lesssim
        \|\hat{f}(t)\|_{L^1_kL^2_{r,\theta}} \lesssim
        \int_\R \e^{-c_1\Lam t}\|\init{\hat{f}}_k\|_{L^2_{r,\theta}}\itd{k}\\
        &\lesssim
        \|\init{\hat{f}}\|_{L^\infty_kL^2_{r,\theta}}\int_\R \e^{-c_1\Lam t}\itd{k}\\
        &\lesssim\norm{\init{f}}_{\LonetLdisc}\int_\R \e^{-c_1\Lam t}\itd{k}.
      \end{aligned}
    \end{equation}
    Then, we control the integral above by splitting the domain of integration in two regions, namely  $\abs{k}\le\nu$ and $|k|>\nu$ which is where the definition of $\Lambda_{\nu,k}$ changes. In particular, we have 
    \begin{align}
    \int_\R \e^{-c_1\Lam t}\itd{k}&=\int_{\abs{k}\le\nu} \e^{-c_1\nu\inv k^2 t}\itd{k}+\int_{\abs{k}>\nu} \e^{-c_1\nu^{\frac{m}{m+2}}\abs{k}^{\frac{2}{m+2}} t}\itd{k}:=\mathcal{I}_{\leq \nu}+\mathcal{I}_{>\nu}.
    \end{align}
    For the low-frequency region, a change of variables shows that
    \begin{equation}
      \label{eqn:physical-int-near}
      \mathcal{I}_{\leq \nu}
      =
      \sqrt{\frac{\nu}{c_1 t}}
      \int_{\abs{\eta}\le\sqrt{c_1\nu t}}\e^{-\eta^2}\itd{\eta}\\
      \lesssim
       \sqrt{\frac{\nu}{ t}}.
    \end{equation}
    For the second integral, since $\nu^{\frac{m}{m+2}}\abs{\nu}^{\frac{2}{m+2}}=\nu$, we estimate as follows:
    \begin{equation}
      \label{eqn:physical-int-away-1}
      \begin{aligned}
       \mathcal{I}_{> \nu}
        &\le
        \e^{-\half c_1\nu t}
        \int_{\abs{k}>\nu} \e^{-\half c_1\nu^{\frac{m}{m+2}}\abs{k}^{\frac{2}{m+2}}t}\itd{k}\\
        &=
        \e^{-\half c_1\nu t}
        \nu^{-\frac{m}{2}}
        \left(\frac12 c_1 t\right)^{-\frac{m+2}{2}}
        \int_{\abs{\eta}>(\half c_1\nu t)^{\frac{m+2}{2}}} \e^{-\eta^{\frac{2}{m+2}}}\itd{\eta}.
      \end{aligned}
    \end{equation}
    We then have the approximation
    \begin{equation}
      \label{eqn:physical-int-away-2}
      \begin{aligned}
        \int_{\abs{\eta}>(\half c_1\nu t)^{\frac{m+2}{2}}} \e^{-\eta^{\frac{2}{m+2}}}\itd{\eta}
        &\le
        \sum_{n=1}^\infty
        \int_{n(\half c_1\nu t)^{\frac{m+2}{2}}}^{(n+1)(\half c_1\nu t)^{\frac{m+2}{2}}} \e^{-\eta^{\frac{2}{m+2}}}\itd{\eta}\\
        &\lesssim
        ( \nu t)^{\frac{m+2}{2}}\sum_{n=1}^\infty
        \e^{-\half n\nu t}\\
        &\lesssim
        (\thalf \cone\nu t)^{\frac{m+2}{2}}
        \int_0^\infty \e^{-\half \eta c_1\nu t}\itd{\eta}\\
        &\lesssim
        (\nu t)^{\frac{m}{2}}.
      \end{aligned}
    \end{equation}
    Inserting this into \eqref{eqn:physical-int-away-1} yields
    \begin{equation}
      \label{eqn:physical-int-away-3}
             \mathcal{I}_{> \nu}
      \lesssim
      t\inv \e^{-\half\cone\nu t},
    \end{equation}
    which combined with \eqref{eqn:physical-int-near} proves the desired
    result.
\end{proof}

 \section*{Acknowledgments} 
The research of MCZ was supported by the Royal
Society through a University Research Fellowship (URF\textbackslash
R1\textbackslash 191492).
The research of MD was supported by the SNSF Grant 182565, by the Swiss State
Secretariat for Education, Research and lnnovation (SERI) under contract number M822.00034 and by GNAMPA-INdAM  through the grant D86-ALMI22SCROB\_01 acronym DISFLU. 

\bibliographystyle{abbrv}
\bibliography{bibenhanced}

\begin{bibdiv}
\begin{biblist}

\bib{ABN22}{article}{
      author={Albritton, Dallas},
      author={Beekie, Rajendra},
      author={Novack, Matthew},
       title={Enhanced dissipation and {H}\"{o}rmander's hypoellipticity},
        date={2022},
        ISSN={0022-1236},
     journal={J. Funct. Anal.},
      volume={283},
      number={3},
       pages={Paper No. 109522, 38},
         url={https://doi.org/10.1016/j.jfa.2022.109522},
      review={\MR{4413303}},
}

\bib{Ar}{article}{
      author={Aris, R},
       title={On the dispersion of a solute in a fluid flowing through a tube},
        date={1956},
     journal={Proc. Roy. Soc. London A},
      volume={235},
       pages={67\ndash 77},
}

\bib{BCZ17}{article}{
      author={Bedrossian, Jacob},
      author={Coti~Zelati, Michele},
       title={Enhanced dissipation, hypoellipticity, and anomalous small noise
  inviscid limits in shear flows},
        date={2017},
        ISSN={0003-9527},
     journal={Arch. Ration. Mech. Anal.},
      volume={224},
      number={3},
       pages={1161\ndash 1204},
         url={https://doi.org/10.1007/s00205-017-1099-y},
      review={\MR{3621820}},
}

\bib{bedrossian2022taylor}{article}{
      author={{Bedrossian}, Jacob},
      author={{Coti Zelati}, Michele},
      author={{Dolce}, Michele},
       title={{Taylor dispersion and phase mixing in the non-cutoff Boltzmann
  equation on the whole space}},
        date={2022-11},
     journal={arXiv e-prints},
      eprint={2211.05079},
}

\bib{CKRZ}{article}{
      author={Constantin, P.},
      author={Kiselev, A.},
      author={Ryzhik, L.},
      author={Zlato\v{s}, A.},
       title={Diffusion and mixing in fluid flow},
        date={2008},
     journal={Ann. of Math. (2)},
      volume={168},
      number={2},
       pages={643\ndash 674},
}

\bib{CZD20}{article}{
      author={Coti~Zelati, Michele},
      author={Dolce, Michele},
       title={Separation of time-scales in drift-diffusion equations on
  {$\Bbb{R}^2$}},
        date={2020},
        ISSN={0021-7824},
     journal={J. Math. Pures Appl. (9)},
      volume={142},
       pages={58\ndash 75},
         url={https://doi.org/10.1016/j.matpur.2020.08.001},
      review={\MR{4149684}},
}

\bib{gallay-cz}{article}{
      author={{Coti Zelati}, Michele},
      author={{Gallay}},
       title={{Enhanced dissipation and Taylor dispersion in higher-dimensional
  parallel shear flows}},
        date={2021-08},
     journal={arXiv e-prints},
      eprint={2108.11192},
}

\bib{FMN23}{article}{
      author={Feng, Yuanyuan},
      author={Mazzucato, Anna~L.},
      author={Nobili, Camilla},
       title={Enhanced dissipation by circularly symmetric and parallel pipe
  flows},
        date={2023},
        ISSN={0167-2789},
     journal={Phys. D},
      volume={445},
       pages={Paper No. 133640, 13},
         url={https://doi.org/10.1016/j.physd.2022.133640},
}

\bib{helffer2010resolvent}{article}{
      author={{Helffer}, Bernard},
      author={{Sjoestrand}, Johannes},
       title={{From resolvent bounds to semigroup bounds}},
        date={2010-01},
     journal={arXiv e-prints},
      eprint={1001.4171},
}

\bib{Ta1}{article}{
      author={Taylor, G.I.},
       title={Dispersion of soluble matter in solvent flowing slowly through a
  tube},
        date={1953},
     journal={Proc. Roy. Soc. London A},
      volume={219},
       pages={186\ndash 203},
}

\bib{Ta2}{article}{
      author={Taylor, G.I.},
       title={Dispersion of matter in turbulent flow through a tube},
        date={1954},
     journal={Proc. Roy. Soc. London A},
      volume={223},
       pages={446\ndash 468},
}

\bib{wei}{article}{
      author={Wei, Dongyi},
       title={Diffusion and mixing in fluid flow via the resolvent estimate},
        date={2021},
        ISSN={1674-7283},
     journal={Sci. China Math.},
      volume={64},
      number={3},
       pages={507\ndash 518},
         url={https://doi.org/10.1007/s11425-018-9461-8},
      review={\MR{4215997}},
}

\end{biblist}
\end{bibdiv}
\end{document}